\documentclass[10pt]{article}
\usepackage{graphicx}
\usepackage{csquotes}
\usepackage{epsfig}
\usepackage{lscape}
\usepackage{amsfonts}
\usepackage[misc,geometry]{ifsym}
\usepackage{color}
\usepackage{amssymb,amsmath}
\usepackage{listings}
\usepackage{rotating}
\usepackage{graphicx}
\usepackage{multicol}
\usepackage[english]{babel}
\usepackage{lineno}
\usepackage{mathrsfs}
\usepackage[pdfpagelabels=true,plainpages=false,colorlinks=true,linkcolor=blue,citecolor=blue,urlcolor=blue]{hyperref}
\date{}
\makeatletter \oddsidemargin  -0.5cm \evensidemargin -0.5cm \textwidth
18.5cm \topmargin 0.005cm \textheight 24cm
\newcommand{\doublespacing}{\let\CS=\@currsize\renewcommand{\baselinestretch}{1.35}\tiny\CS}
\usepackage{amssymb}
\usepackage{amsmath}
\usepackage{amsthm}
\usepackage{enumerate}
\usepackage{float}
\restylefloat{table}
\usepackage{caption}
\usepackage{mathtools}
\usepackage{lipsum}
\usepackage{xcolor}
\usepackage{soul}
\sethlcolor{black}
\makeatletter
\let\chapter\section
\usepackage[ruled,vlined]{algorithm2e}
\usepackage{afterpage}
\usepackage{url}
\usepackage{geometry}
\usepackage{pdflscape}
\usepackage{cases}
\usepackage{rotating}
\usepackage{booktabs}
\usepackage{tikz}

\usepackage[numbers,sort&compress]{natbib}% Citation support using natbib.sty
\bibpunct[, ]{[}{]}{,}{n}{,}{,}% Citation support using natbib.sty
\makeatletter% @ becomes a letter
\def\NAT@def@citea{\def\@citea{\NAT@separator}}% Suppress spaces between citations using natbib.sty
\makeatother% @ becomes a symbol again
\usepackage{epstopdf}
\theoremstyle{plain}% Theorem-like structures provided by amsthm.sty
\newtheorem{theorem}{Theorem}[section]

\usepackage{placeins}
\usepackage[ruled,vlined]{algorithm2e}
\newtheorem{definition}{Definition}
\newtheorem{prop}{Proposition}

\newtheorem{ex}{Example}
\newtheorem{remark}{Remark}
\newtheorem{notation}{Notation}
\usepackage{mathrsfs}
\usepackage{epsfig}
\begin{document}
\markboth{P.Roy, G.Panda}
{Gradient based line search scheme for interval valued optimization problem}

%\catchline{}{}{}{}{}
%%%%%%%%%%%%%%%%%%%%%%%%%%%%%%%%%%%%%%%%%%%%%%%%%%%%%%%%%%%%%%%%%%%%

\title{On Critical Point for Functions with Bounded Parameters}

\author{Priyanka Roy$ ^\ddagger $  and Geetanjali Panda$^*$
%\footnote{Typeset names 
%in 8 pt roman, uppercase. Use the footnote to 
%indicate the present or permanent address of the author.}
\\
 Department of Mathematics, Indian Institute of Technology Kharagpur, \\
Kharagpur, West Bengal-721302, India.\\
$ ^\ddagger $ proy180192@maths.iitkgp.ac.in\\
$^*$ geetanjali@maths.iitkgp.ac.in}

\maketitle

%\titlerunning{Gradient based line search scheme for interval valued optimization problem}        % if too long for running head

%\section*{Gradient based line search scheme for interval valued optimization problem}
%\author{\hl{Priyanka Roy        \and
%        Geetanjali Panda
%}}
%
%%\authorrunning{Short form of author list} % if too long for running head
%
%\institute{\hl{P Roy  \and
%           G Panda \at
%              Department of Mathematics, Indian Institute of Technology, Kharagpur - 721302, India \\
%              \email{geetanjali@maths.iitkgp.ernet.in}
%\maketitle

%\begin{abstract}
\begin{abstract}
Selection of descent direction at a point plays an important role in numerical optimization for minimizing a real valued function. In this article, a descent sequence is generated for the functions with bounded parameters to obtain a critical point. First, sufficient condition for the existence of descent direction is studied for this function and then  a set of descent directions at a point is determined using linear expansion. Using these results a descent sequence of intervals is generated and critical point is characterized. This theoretical development is justified with numerical example.\\ \\
\textbf{keywords}
Line search technique, Markov difference, Interval ordering, Descent direction, Critical point.  

\end{abstract}

\section{Introduction}
In most of the mathematical models, the parameters vary in some bounds which can be estimated from historical data. These uncertain parameters can be considered as intervals. In that case, the functions involved in the model here known as  functions with bounded parameters as intervals. Interval analysis plays an important role to handle these functions. This article has studied some properties of these type functions,
         $\breve{F}$ from $\mathbf{R}^n$ to the set of closed intervals, whose parameters are intervals. An example of such type function is \\ $\breve{F}(x,y)=e^{[-3,3]x^3}\oplus [4,5] xy^5$. In the literature of interval analysis, Markov is the  pioneer who has studied different areas of modern mathematics using interval analysis ( see \cite{dimitrova1992extended}, \cite{Kyurkchiev2016}, \cite{MARKOV1999225}, \cite{markov2000algebraic}, \cite{MARKOV200493}, \cite{markov2005quasilinear} etc. ). Markov has introduced nonstandard substraction between two closed intervals( known as Markov difference, $\ominus_M$),  which has explored calculus of functions with bounded parameters and its application in several areas in recent times (see  \cite{bhurjee2016optimality}, \cite{bhurjee2012efficient}, \cite{markov1979calculus}, \cite{MARKOV2006}, \cite{MARKOV2015} etc.). Another nonstandard difference, known as gH-difference,  in the set of intervals is introduced in \cite{stefanini2008generalization} and further developed in \cite{chalco2013calculus}, \cite{Lupulescu201350},\cite{stefanini2009generalized} etc. . It is justified in \cite{chalco2011generalized} that gH difference coincides with Markov difference in case of compact intervals. Markov difference is more comfortable for use in numerical computations. Hence in this article we have accepted Markov difference and proceed to develop an iterative process for generating a descent sequence of intervals. This concept may be extended further to develop a descent sequence of points which may converge to a local minimum point of a function with interval parameters under reasonable conditions. This can explore a new area of numerical optimization, which may be considered as the possible scope of the present contribution. At this present stage we focus on characterizing descent direction, generating descent sequence of intervals for a  function with interval parameters, which provides the critical point of the function. \par 
       Some prerequisites on interval analysis are discussed in Section 2. In Section 3, Markov difference is used to derive the linear expansion  of $\breve{F}$ and existence of descent direction of $\breve{F}$ at a point $x$ is studied. Section 4 is devoted for generating descent sequence of intervals which determines the critical point of $\breve{F}$. Section 5 provides concluding remarks with future scope.

\section{Prerequisites}
$ K(\textbf{R}) $ denotes the set of all compact intervals on $ \textbf{R} $ throughout this article . $ \breve{\alpha} \in K(\mathbf{R}) $ is the closed interval of the form $ [\underline{\alpha},\overline{\alpha}] $ where $ \underline{\alpha} \leq \overline{\alpha} $. For two points $\alpha_{1}$ and $ \alpha_{2} $,(not necessarily $\alpha_{1}\leq \alpha_{2}  $), $ \breve{\alpha} $ can be written as $  \breve{\alpha}=[\alpha_{1} \vee \alpha_{2}]$.  A real number $r$ can be represented as a degenerate interval denoted by $\breve{r}$ as $\breve{r}=[r,r] $~or $r.\breve{I}$, where $\breve{I}=[1,1]$. The null interval is $\breve{0}=[0,0]=0$.
\\ In $ K(\mathbf{R}) $, the norm $( \Vert .\Vert )$ of an interval $ \breve{\alpha} $ is defined as $ \Vert \breve{\alpha} \Vert= \max\left\lbrace | \underline{\alpha}|,| \overline{\alpha}|\right\rbrace  $ (\cite{markov1977extended} ) which is associated with the metric $ d(\breve{\alpha},0)=  \Vert \breve{\alpha}\Vert $ and $ d(\breve{\alpha}, \breve{\beta})=\max\left\lbrace | \underline{\alpha}-\underline{\beta} |,~| \overline{\alpha}-\overline{\beta}| \right\rbrace $.
 \\
 $K(\mathbf{R})$ is not a complete ordered set. Following interval ordering is used throughout the article..\\
  For $\breve{\alpha}$, $\breve{\beta}\in K(\mathbf{R})$, $ \breve{\alpha}\preceq \breve{\beta}\Leftrightarrow \underline{\alpha} \leq \underline{\beta}$, $\overline{\alpha}\leq \overline{\beta} $ and $\breve{\alpha}\neq \breve{\beta}$;
 $\breve{\alpha}\prec \breve{\beta} \Leftrightarrow \underline{\alpha} < \underline{\beta}$ and $\overline{\alpha}< \overline{\beta}$.

The other interval order relations $`\succeq '$ and $ '\succ' $ can be defined in a similar way.
\\
%   Algebraic  operation between two intervals  $ \breve{\alpha} $, $ \breve{\beta} $  is defined as $  \breve{\alpha} \circledast \breve{\beta}=\left\lbrace a* b | a\in \breve{\alpha},b\in \breve{\beta}  \right\rbrace,$ where $ * \in \left\lbrace +,-,\cdot,/ \right\rbrace  $. $ K(\mathbf{R}) $ is closed with respect to these operations.

 Additive inverse in $ \left\langle K(\mathbf{R}),\oplus,\odot\right\rangle  $ may not exist, that is, $ \breve{\alpha}\ominus \breve{\alpha} $ is not necessarily $\breve{0}$ according to this approach.
 The non-standard subtraction due to \cite{markov1979calculus}, denoted by $ \ominus_{M} $, provides additive inverse, which is
 \begin{align}\label{difference}
 \breve{\alpha}\ominus_{M}\breve{\beta}=\left[\min\left\lbrace \underline{\alpha}-\underline{\beta},\overline{\alpha}-\overline{\beta}\right\rbrace,\max\left\lbrace \underline{\alpha}-\underline{\beta},\overline{\alpha}-\overline{\beta}\right\rbrace\right]
 \end{align}
 Following properties of $ \ominus_{M} $ due to \cite{dimitrova1992extended} and \cite{markov1979calculus} are used throughout the article.\\
 \textbf{(i) } $ \breve{\alpha}\ominus_{M} \breve{\alpha}= \breve{0} $ ; \textbf{(ii) } $\ominus_{M}\breve{\alpha}=(-1)\breve{\alpha}=[-\overline{\alpha},-\underline{\alpha}] $

%In $ K(\mathbf{R}) $, the norm $( \parallel . \parallel )$ of an interval $ \breve{u} $ is defined as $ \parallel \breve{u} \parallel= \max\left\lbrace \mid \underline{u}\mid,\mid \overline{u}\mid \right\rbrace  $ (\cite{markov1977extended} )and the metric is defined as $ d(\breve{u},0)=  \parallel \breve{u} \parallel $ and $ d(\breve{u}, \breve{v})=\max\left\lbrace \mid \underline{u}-\underline{v} \mid,\mid \overline{u}-\overline{v}\mid \right\rbrace $.
 %\\
% $K(\mathbf{R})$ is not a totally ordered set. Several partial orderings exist in the literature of interval analysis(see  \cite{moore2009introduction}).  Following interval ordering is often used for solving interval optimization problem.\\
%  For $\breve{u}=[\underline{u},\overline{u}]$, $\breve{v}=[\underline{v}$, $\overline{v}]\in K(\mathbf{R})$,
%$ \breve{u}\underset{=}\prec \breve{v} \Leftrightarrow \underline{u} \leq \underline{v}$ and $\overline{u}\leq \overline{v} $;
%$ \breve{u}\preceq \breve{v}\Leftrightarrow \breve{u}\underset{=}\prec \breve{v}$ and $\breve{u}\neq \breve{v}$;
% $\breve{u}\prec \breve{v} \Leftrightarrow \underline{u} < \underline{v}$ and $\overline{u}< \overline{v}$.
%
%The interval order relations  $`\underset{=}\succ '$ $`\succeq '$ and $ '\succ' $ are defined in a similar way by reverting the inequalities.

Limit and continuity  of $ \breve{F} $ are understood in the sense of $ \parallel . \parallel$ due to \cite{markov1979calculus}.
Following results due to \cite{markov1979calculus} are summarized for $ \breve{F}:\mathbf{R}\rightarrow K(\mathbf{R}) $,  $ \breve{F}(x)=[\underline{F}(x),\overline{F}(x)] $, where $ \underline{F},\overline{F}:\mathbf{R}\rightarrow \mathbf{R} $, $ \underline{F}(x)\leq \overline{F}(x) $ $ \forall~~x\in \mathbf{R} $.\\
\begin{definition}[Definition 2, \cite{markov1979calculus}]
\label{Definition 2}
 $ \breve{F}:\mathbf{R}\rightarrow K(\mathbf{R}) $ is differentiable at $  x_0$ if $\displaystyle \lim_{h\rightarrow 0}\frac{\breve{F}(x_0+h) \ominus_{M} \breve{F}(x_0)}{h} $ exists. The limiting value is the derivative of $ \breve{F} $ at $x_0$, denoted by $ \breve{F}^{\prime}(x_0) $.\\
Alternatively if $ \exists \breve{F}^{\prime}(x_0) \in K(\mathbf{R}) $ and an error function \\ $ \breve{E}_{x_0}:\mathbf{R} \rightarrow K(\mathbf{R}) $ at $ x $ such that
\begin{align}
\breve{F}(x_0+h)\ominus_{M} \breve{F}(x_0)= h\odot \breve{F}^{\prime}(x_0)\oplus \breve{E}_{x_0}(h) \label{df-1} \\ 
\nonumber 
\text{ where } \lim_{h \rightarrow 0} \breve{E}_{x_0}(h)=\breve{0} 
\end{align}
\end{definition}

\begin{theorem}[Theorem 9, \cite{markov1979calculus}]
\label{Mean-Value theorem for interval functions}
If $ \breve{F}:\mathbf{R}\rightarrow I( \mathbf{R})$ is \\ continuous in $ \Delta$, where $ \Delta=\left[p,q \right]$ and differentiable in $(p,q ) $,then $ \breve{F}(q)\ominus_{M} \breve{F}(p)\subset \breve{F}^{\prime}(\Delta)(q-p) $, where\\ $ \displaystyle \breve{F}^{\prime}(\Delta)=\cup_{\xi \in \Delta}\breve{F}^{\prime}(\xi)$.
\end{theorem}
% \vspace{0.25cm}
 
Above results  discuss the calculus of interval function on $\mathbf{R}$. In next section some of these results are extended to develop calculus of interval functions on $ \mathbf{R}^{n} $.

\section{Descent direction for interval function over $ \mathbf{R}^n $}
%\begin{definition}
%$ \lim_{x\rightarrow x_{0}} \breve{F}(x)=\breve{F}_{0} $ exists if for arbitrarily chosen $ \epsilon >0 $, $ \exists~~ \delta>0 $ such that $ d(\breve{F}(x),\breve{F}_{0})<\epsilon $ for any $ x\in \mathbf{R}^{n} $ satisfying $ 0<\parallel x-x_{0}\parallel<\delta $.
%\end{definition}
%
%\begin{definition}
%\label{Continuity}
%$ \breve{F}:\mathbf{R}^{n}\rightarrow K(\mathbf{R}) $ is called continuous at $ \overline{x} \in \mathbf{R}^{n} $ if for arbitrarily chosen $ \epsilon > 0 ,\exists \delta > 0$ such that $ d(\breve{F}(x),\breve{F}(\overline{x}))< \epsilon $ for any $x\in \mathbf{R}^{n}$ satisfying $ \parallel x-\overline{x}\parallel<\delta $.
%\end{definition}
%Following result follows directly from this definition.
%\begin{prop}
%\label{pro}
%  $ \breve{F} $ is continuous at a point iff $ \underline{F} $ and $ \overline{F} $ are continuous at that point.\end{prop}
\begin{definition}
\label{Partial derivatives}
For $ \breve{F}:\mathbf{R}^{n}\rightarrow K(\mathbf{R}) $, the partial derivative \\ of $ \breve{F} $ with respect to $ x_{i} $ at `$  x $ ' exists   if $ \underset{h_i\rightarrow 0}\lim \dfrac {\breve{F}(x_1, \cdots, x_{i-1}, x_{i}+h_i, x_{i+1}, \cdots x_{n})\ominus_{M}\breve{F}(x)}{h_i} $ exists. The limiting value is denoted by $ \frac{\partial \breve{F}(x)}{\partial x_{i}} $.
\end{definition}
%\vspace{0.25cm}
In the light of concept of differentiability in \eqref{df-1} of Definition \ref{Definition 2}, the following can be stated as follows.
%\vspace{0.15cm}
\begin{definition}
\label{Differentiability}
 $ \breve{F}:\mathbf{R}^{n}\rightarrow K(\mathbf{R}) $  is called differentiable at $x_0$ if $ \frac{\partial \breve{F}(x_0)}{\partial x_i} $ exists $ \forall~i $ and an interval function \\ $ \breve{E}_{x_0}(h):\mathbf{R}^{n}\rightarrow K(\mathbf{R}) $ such that
$$ \breve{F}(x_0+h)\ominus_{M} \breve{F}(x_0)=\sum\limits_{i=1}^{n}\left( \oplus h_{i}\odot \frac{\partial \breve{F}(x_0)}{\partial x_{i} }\right) \oplus \|h\| \odot  \breve{E}_{x_0}(h)  $$
for $ \|h\| < \delta $  for some $ \delta>0 $, where $\lim_{\|h\|\rightarrow 0} \breve{E}_{x_0}(h)= \breve{0}$.
\end{definition}
%\vspace{0.25cm}
Gradient of $ \breve{F} $ at $ x $ is  $(\frac{\partial \breve{F}(x)}{\partial x_{1}},\frac{\partial \breve{F}(x)}{\partial x_{2}},\cdots,\frac{\partial \breve{F}(x)}{\partial x_{n}})^{\prime}$ and denoted by $ \nabla \breve{F} (x)$.\\

Following result is about linear expansion of $ \breve{F} $ in inclusion form which will be used further to derive descent direction.
\begin{theorem}\label{mean value theorem}
Let $\Omega$ be an open convex subset  of $ \mathbf{R}^{n} $ and $ \breve{F}:\Omega \rightarrow K(\mathbf{R}) $ be differentiable  on $\Omega$. Then for any $ u,~v \in \Omega $,
\begin{align}
\breve{F}(v) \ominus_{M} \breve{F}(u)\subset \underset{c \in L.S\lbrace u,v \rbrace}\cup \sum_{i=1}^{n}(v_i-u_i)\odot\frac{\partial \breve{F}(c)}{\partial \gamma_i} \label{mean}
\end{align}
where $ \gamma(t)= u+t (v-u)$, $ t \in [0,1] $ and $ L.S\lbrace u,v \rbrace  $ denotes the line segment joining $ u $ and $ v $.
\end{theorem}
\begin{proof} Since $ \Omega $ is convex subset of $ \mathbf{R}^{n} $, for $ u, v \in \Omega $, $ u+t(v-u) $ with $ t \in [0,1] $  must belongs to $ \Omega $.\\
Let $\breve{\phi}:[0,1] \rightarrow K(\mathbf{R})$ is defined by \\ $ \breve{\phi}(t)= \breve{F}( \gamma_1(t), \gamma_2(t), \cdots \gamma_n(t)) $.  Since $\breve{F}(\gamma)$ and $ \gamma(t) $ are differentiable, using Definition \ref{Differentiability} and first order Taylor expansion of  $ \gamma $ it can be shown that their composite function $ \breve{\phi} $ is differentiable and 
\\ $ \breve{\phi}^{\prime}(t)= \displaystyle \sum_{i=1}^{n} \gamma_i^{\prime}(t) \odot \frac{\partial \breve{F} (\gamma(t))}{\partial \gamma_{i}}= \displaystyle \sum_{i=1}^{n}(v_i-u_i)\odot \frac{\partial \breve{F}(\gamma(t))}{\partial \gamma_{i}}  $.\\\\
%$ \breve{\phi} $ is differentiable on $[0,1]  $.
%From Theorem \ref{Theorem 8}, $ \breve{\phi}^{\prime}(t)= (b-a)^{T}\nabla \breve{F}(\gamma (t)) $
 From Theorem \ref{Mean-Value theorem for interval functions},
$
\breve{\phi}(1)\ominus_{M} \breve{\phi}(0)\subset \underset{\theta \in [0,1]}\cup \breve{\phi}^{\prime}(\theta)$.
 Here $ \breve{\phi}(1)=\breve{F}(v) $ and $ \breve{\phi}(0)= \breve{F}(u) $. Hence \eqref{mean} follows, where $ c= a+\theta(b-a) $ for some $\theta$ and  $ \breve{\phi}^{\prime}(\theta)= \displaystyle\sum_{i=1}^{n}(v_i-u_i)\odot\frac{\partial \breve{F}(c)}{\partial \gamma_i} $. 
  \end{proof}
\begin{prop}\label{lm}
Let $ \breve{\alpha},\breve{\beta}\in K(\mathbf{R}) $. Then $ \breve{\alpha}\prec \breve{\beta}$ holds if and only if $ \breve{\alpha}\ominus_{M}\breve{\beta}\prec \breve{0} $.
\end{prop}
%\vspace{0.15cm}
Proof of this result is straight forward from the definition of $\ominus_{M}$.

\begin{definition}
\label{descent direction}
$ d\in \mathbf{R}^{n} $ is called a descent direction of $ \breve{F} $ at $ x $ if $ \exists $ some $ \delta>0 $ so that $ \breve{F}(x+\alpha d)\prec \breve{F}(x)~~\forall \alpha\in (0,\delta)$.
\end{definition}
%\vspace{0.25cm}
\begin{notation} 
\label{not} 
Denote $\nabla\breve{F}(x)$ by $\breve{g}(x)$ where \\ $\breve{g}(x)=\begin{pmatrix} \breve{g}_{1}(x) & \breve{g}_{2}(x) & \cdots &  \breve{g}_{n}(x)\end{pmatrix}^{\prime}$, with \\ $\breve{g}_{i}(x)=[\underline{g}_{i}(x), \overline{g}_{i}(x)]~\forall i= 1,2, \cdots, n $ where \\ $\underline{g}_{i}(x)=\min \lbrace  \frac{\partial \underline{F}(x)}{\partial x_{i}}, \frac{\partial \overline{F}(x)}{\partial x_{i}}\rbrace,~ {\overline{g}_{i}(x)=\max \lbrace \frac{\partial \overline{F}(x)}{\partial x_{i}}, \frac{\partial \underline{F}(x)}{\partial x_{i}} \rbrace}$.\\
Denote $\underline{g}(x)\triangleq \begin{pmatrix} \underline{g}_{1}(x)& \underline{g}_{2}(x),& \cdots,& \underline{g}_{n}(x) \end{pmatrix}^{\prime}$ and \\
 $ \overline{g}(x)\triangleq \begin{pmatrix} \overline{g}_{1}(x)& \overline{g}_{2}(x),& \cdots, &\overline{g}_{n}(x)
\end{pmatrix}^{\prime}  $.
%\underline{g}_{1}(x), \underline{g}_{2}(x),...,\underline{g}_{n}(x))^{\prime},~~
%\overline{g}(x)\triangleq (\overline{g}_{1}(x), \overline{g}_{2}(x),...,\overline{g}_{n}(x))^{\prime}
\end{notation}
%\vspace{0.25cm}
\begin{theorem}\label{desdir}
Let $ \breve{F}:\Omega\subseteq \mathbf{R}^{n}\rightarrow K(\mathbf{R}) $ be continuously differentiable. If $ \displaystyle \sum_{i=1}^n d_i \odot \breve{g}_{i}(x)\prec \breve{0} $, then $d$ is a descent direction of $ \breve{F} $ at $x$.
\end{theorem}
\begin{proof}
$\displaystyle \sum_{i=1}^n d_i \odot \breve{g}_{i}(x)$ is continuous since $  \breve{g} $ is continuous. Since  $ \displaystyle \sum_{i=1}^n d_i \odot \breve{g}_{i}(x)\prec \breve{0} $, so  $ \displaystyle 0 \notin  \sum_{i=1}^n d_i \odot \breve{g}_{i}(x) $. Therefore $ \exists~~ \delta>0 $ such that
\begin{equation}\label{ee}
 \sum_{i=1}^n d_i \odot \breve{g}_{i}(x+ \alpha d)\prec \breve{0} ~~\forall \alpha\in (0,\delta).
\end{equation}
Using  Theorem \ref{mean value theorem},
\begin{align}
\breve{F}(x+\alpha d)\ominus_{M} \breve{F}(x)\subset \underset{c\in~ L.S \lbrace x,x+\alpha d\rbrace } \cup \alpha \sum_{i=1}^n d_i \odot \breve{g}_{i}(c)\label{eq 16}
\end{align}
From  \eqref{ee}, $ \displaystyle \sum_{i=1}^n d_i \odot \breve{g}_{i}(c)\prec \breve{0} $, for each $ c \in L.S\left \lbrace x,x+\alpha d \right\rbrace  $. Hence from  \eqref{eq 16}, $\breve{F}(x+\alpha d)\ominus_{M} \breve{F}(x)\prec \breve{0}$.\\ $  \breve{F}(x+\alpha d)\prec \breve{F}(x) $ follows from Proposition \ref{lm}. Therefore $d$ is the descent direction at $ x $. 
\end{proof}
	
%to the contours  the the depicted region is the set of descent directions at $(x^0,y^0 )$ in XY plane
%
% we tried to show the set of descent directions in XY plane as projection. At the reference point $ (x_0 ,y_0) $, $ Z_1 =\underline{F}(x_0 ,y_0)$ and  $ Z_2 =\overline{F}(x_0 ,y_0)$. Single contour line of $ \underline{F}(x,y) $ with $ Z_1 $ level and $ \overline{F}(x,y) $ with $ Z_2 $ level will pass through $ (x_0 ,y_0) $. Intersection of the open half spaces generated by tangent lines to $ \underline{F}(x,y) $ and $ \overline{F}(x,y) $ at the reference point $ (x_0 ,y_0) $ indicates the set of descent directions satisfying \eqref{d1}.

\begin{figure}[h!]
\begin{center}
%\begin{framed}
%\hspace{5cm}
\includegraphics[scale=0.55]{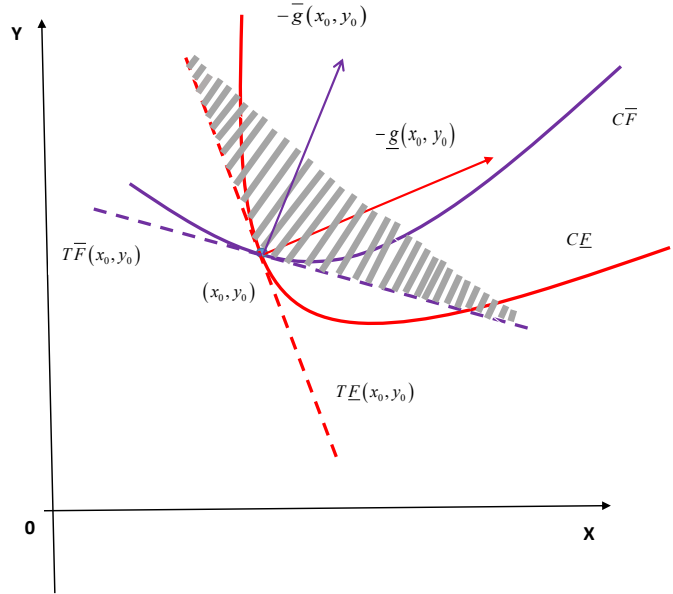}
%\vspace{-2cm}
\caption{Shaded region indicates the set of descent directions}
\label{fig1}
%\end{framed}
\end{center}
\end{figure}
\FloatBarrier
From the Theorem \ref{desdir}, one may conclude that the descent direction of  $\breve{F}(x,y)= [\underline{F}(x,y),\overline{F}(x,y)]     $ at point $(x_0 ,y_0)$  can be determined by solving  $  \sum_{i=1}^n d_i \odot \breve{g}_{i}(x_0 ,y_0)\prec \breve{0} $ i.e. $ \breve{g}(x_0 ,y_0)^{\prime} \odot d\prec \breve{0}      $. The set of descent directions is
\begin{equation*}
	 \left\lbrace d \in \mathbf{R}^{n}: \underline{g}(x_0,y_0)^{\prime}d<0 ; \overline{g}(x_0,y_0)^{\prime}d<0                           \right\rbrace.
\end{equation*}
	In Figure \ref{fig1}, $C\underline{F}$ and $C\overline{F}$ are the contours of lower and upper bound functions at levels $\underline{F}(x_0 ,y_0)$ and  $\overline{F}(x_0 ,y_0)$ respectively. $T\underline{F}$ and $T\overline{F}$ are the tangent lines to the contours $C\underline{F}$ and $C\overline{F}$ at $(x_0 ,y_0)$ respectively. The set of descent directions at $(x_0 ,y_0)$ is the set of vectors in $R^2$, which make obtuse angle with both $\underline{g}(x_0 ,y_0)$ and $\overline{g}(x_0 ,y_0)$. This is the shaded region in Figure \ref{fig1}.
\section{Descent direction and generating descent sequence}
Selection of suitable descent direction plays an important role while developing an efficient numerical algorithm for minimizing a real valued function $f(x)$. Objective of this section is to develop a descent sequence $\breve{F}(x^k)$  with respect to the interval ordering  $\preceq $, staring at an initial point $x^0$.  For the descent direction $d^k\in \mathbf{R}^n$ at $x^k$ of $\breve{F}(x)$, $\breve{F}(x^{k+1})\preceq \breve{F}(x^{k})$ holds causing reduction on lower and upper function simultaneously. Here $x^{k=1}=x^k+\alpha_kd^k$, where $\alpha_k>0$ is the step length at $x^k$ in the descent direction $d^k$, selected in such a way that $\breve{F}(x^{k+1})\preceq \breve{F}(x^{k})$ holds. This process terminates at a point $x^*$, when either $\alpha_k=0$ or no such descent direction exists. We say a point where no such descent direction exists as a critical point of $\breve{F}(x)$, which is defined below.
%
%
%
%
%
%
%
%
%for the interval function $\breve{F}(x)$ which can play an important role in future to develop a suitable line search method for minimizing an interval function.\par
% Since $ K(\mathbf{R}) $ is not a totally ordered space, so minimization of $\breve{F}(x)$ should be understood with respect to the partial orderings in $ K(\mathbf{R}) $. The readers may see \cite{osuna2015optimality} ****ajaya*** for the concept of minimum of interval function(Definition 6), which is stated below. %Since $\breve{F}$ is a set valued mapping so solution of $(\min \hat{P})$ is called as efficient point. \\
%A point $x^{*}\in \mathbf{R}^{n}$ is called an efficient point of $\breve{F}$ if $ \nexists~x\in \mathbf{R}^{n} $ such that $ \breve{F}(x)\preceq \breve{F}(x^{*}) $.
%A point $x^{*}\in \mathbf{R}^{n}$ is called a weak  efficient point of $\breve{F}$ if $ \nexists~x\in \mathbf{R}^{n} $ such that $ \breve{F}(x)\prec \breve{F}(x^{*}) $. Therefore in the light of descent direction, for a weak efficient point $ x^* $, $ \nexists~d \in \mathbf{R}^{n}\mbox{ such that}~~ p^{T}d<0 ~~\forall~p\in  \breve{g}(x^{*})$.
\begin{definition}\label{critic}
A point $ x^* $ is called a critical point of $ \breve{F} $ if  $ \nexists~d \in \mathbf{R}^{n}\mbox{ such that}~~ p^{T}d<0 ~~\forall~p\in  \breve{g}(x^{*})$ .
\end{definition}
\begin{remark} Weak efficient solution of the optimization problem $\min_{x\in R^n} \breve{F}(x)$ is a critical point of $\breve{F}(x)$. For a detailed study of weak efficient solution, the readers may see \cite{bhurjee2012efficient} and \cite{osuna2015optimality}.
\end{remark}
%In this section the descent sequence $\breve{F}(x^k)$  is generated with respect to $\preceq $ ordering of intervals using the concepts developed in previous sections. For the descent direction $d^k\in \mathbf{R}^n$ at $x^k$ of $\breve{F}(x)$, $\breve{F}(x^{k+1})\preceq \breve{F}(x^{k})$ holds causing reduction on lower and upper function simultaneously.
Following results are studied in this direction, which may be considered as a stepping stone for generating descent sequence.
\begin{theorem}\label{ds}
Let $ \breve{F}:\Omega\rightarrow  K(\mathbf{R}) $ be differentiable such that $ 0 \notin int(\breve{g}_{i}(x))~\forall i=1,2, \cdots, n $. Then  $\overline{d}$ is a descent direction at $  x $ where $ \overline{d}_i \in \ominus_{M} \breve{g}_{i}(x) $ for each $ i=1,2, \cdots, n $ .
\end{theorem}
\begin{proof}
Since  $  0 \notin int(\breve{g}_{i}(x))\forall i= 1,2, \cdots, n $, either $ \frac{\partial \breve{F}}{\partial x_{i}}\succeq \breve{0} $ or $ \frac{\partial \breve{F}}{\partial x_{i}}\preceq \breve{0}\forall i=1,2, \cdots, n $. \\
Since $ \overline{d}_i \in \ominus_{M} \breve{g}_{i}(x) $ for each $ i=1,2, \cdots, n $, so $ d_{i}\geq  0  $ if $ \frac{\partial \breve{F}}{\partial x_{i}} \preceq  \breve{0}$ and vice versa. \\
So, $ \sum_{i=1}^{n} \overline{d}_i \odot \frac{\partial \breve{F}(x)}{\partial x_{i}} \prec \breve{0} $.\\
Therefore using Theorem \ref{desdir}, $\overline{d}$ is descent direction at $ x $. 
\end{proof}

\begin{remark}\label{dis}
From the above theorem the following results can be concluded.
\begin{itemize}
\item[$\bullet$] For a descent direction $ \overline{d} $, since $\overline{d}_{i}\in\ominus_{M}\breve{g}_{i}(x)=[-\overline{g}_{i}(x), -\underline{g}_{i}(x)] $ for each $ i=1,2, \cdots, n $ , each $ d_i $ can be written as $ d_i = -\overline{g}_{i}(x)+t_i (\overline{g}_{i}(x)-\underline{g}_{i}(x))  $ for each $ t_i \in [0,1] $.
\item[$\bullet$] If $ \exists $ some $ i $, for which $ 0 \in int(\breve{g}_{i}(x)) $, then it may not guarantee that $ \overline{d} $ is a descent direction at $ x $. In that case for preserving descent property one may choose $ \overline{d}_i=0 $ for those $ i $.
\end{itemize}
\end{remark}

\subsection*{Generating a descent sequence of interval function}
For generating descent sequence of intervals $\{\breve{F}(x^k)\}$ satisfying $\breve{F}(x^{k+1})\preceq \breve{F}(x^k)$ and \\ $x^{k+1}=x^k+\alpha_{k}d^k$, selection of $\alpha_{k}$ and suitable stopping condition play important role. The iterative process will be stopped at critical point.  Therefore a tolerance level $ \epsilon > 0 $ can be fixed for $ \parallel d^k \parallel $ as a breaking condition.
Here exact line search technique for real valued functions is used for step-length selection in this iterative process, which is explained below.\\
Suppose
$ \underline{\alpha}_{k}= \arg ~\underset{ \alpha>0} \min \underline{F}(x^{k}+\alpha d^{k}) $ and\\$ \overline{\alpha}_{k}= \arg~ \underset{ \alpha>0} \min \overline{F}(x^{k}+\alpha d^{k}) $ .
Choose
\begin{equation}\label{alp}
\alpha_{k}= \min \lbrace \underline{\alpha}_{k},  \overline{\alpha}_{k}         \rbrace
\end{equation}
 along the descent direction $ \lbrace      d^{k}   \rbrace $.
In particular if $\underline{F}(x^{k}+\alpha d^{k})$ and $ \overline{F}(x^{k}+\alpha d^{k})$ are convex functions in $\alpha $ then $ \underline{\alpha}_{k}= \min \lbrace  \alpha>0 :  \underline{g}(x^{k}+\alpha d^{k})^{\prime}d^{k}=0   \rbrace $ and\\
$ \overline{\alpha}_{(k)}= \min \lbrace  \alpha>0 : \overline{g}(x^{k}+\alpha d^{k})^{\prime}d^{k}=0                 \rbrace $\\
 The above theoretical development is summarized in following steps.
\begin{enumerate}
\item Fix tolerance level $ \epsilon $, initialize $ k=0 $, $ x^0 $, $ \alpha_0 $, fix $ t _i \in$ rand $[0,1] $.
\item Choose $ \overline{d}_i (t)$ according to Remark \ref{dis}.
\item Select suitable step length $ \alpha_k $ according to \eqref{alp}.
\item $ x^{k+1}= x^{k}+ \alpha_k \overline{d}^{k} $.
\item Check if $\parallel \overline{d}^{k+1} \parallel  < \epsilon $, stop. Otherwise  $ k :=k+1 $.
\end{enumerate}

The above steps can be verified with the following example. In this example a descent sequence is generated and critical point is verified.
\begin{ex} \label{num2}
 Consider \\ $ \breve{F}(x_{1},x_{2})=[2,4]x_{1}^{2} \oplus [2,3]x_{1}x_{2} \oplus [1,2]x_{2}^{2}\oplus [1,2]x_{1}\ominus_{M}[1,3]x_{2} $. \\ \\
 
 $\breve{F}(x_{1},x_{2})= \begin{cases}
 \left[ 2x_1^2 +2x_1x_2+x_2^2+x_1-x_2 \right. \vee \\ \left. 4x_1^2 +3x_1x_2+2x_2^2+2x_1-3x_2\right];\\ \text { if } x_{1} \geq 0,x_{2}\geq 0\\
 \left[ 2x_1^2 +3x_1x_2+x_2^2+2x_1-x_2 \right. \vee \\ \left. 4x_1^2 +2x_1x_2+2x_2^2+x_1-3x_2\right];\\ \text { if } x_{1} \leq 0,x_{2}\geq 0\\
  \left[ 2x_1^2 +2x_1x_2+x_2^2+2x_1-3x_2 \right. \vee \\ \left. 4x_1^2 +3x_1x_2+2x_2^2+x_1-x_2\right];\\ \text { if } x_{1} \leq 0,x_{2}\leq 0\\
  \left[ 2x_1^2 +3x_1x_2+x_2^2+x_1-3x_2 \right. \vee \\ \left. 4x_1^2 +2x_1x_2+2x_2^2+2x_1-x_2\right];\\ \text { if } x_{1} \geq 0,x_{2}\geq 0
 \end{cases} $
 \\ \\
 $\nabla \breve{F}(x_{1}, x_{2})=\begin{cases}
 \begin{pmatrix}
 [4x_{1}+2x_{2}+1 \vee 8x_{1}+3x_{2}+2] \\ [2x_{1}+2x_{2}-1 \vee 3x_{1}+4x_{2}-3]
 \end{pmatrix};\\ \text { if } x_{1} \geq 0,x_{2}\geq 0\\
 \begin{pmatrix}
 [  4x_{1}+3x_{2}+2 \vee  8x_{1}+2x_{2}+1] \\ [3x_{1}+2x_{2}-1 \vee  2x_{1}+4x_{2}-3]
 \end{pmatrix};\\ \text { if } x_{1} \leq 0,x_{2}\geq 0\\
 \begin{pmatrix}
 [ 4x_{1}+2x_{2}+2 \vee 8x_{1}+3x_{2}+1] \\ [ 2x_{1}+2x_{2}-3 \vee   3x_{1}+4x_{2}-1 ]
 \end{pmatrix};\\ \text { if } x_{1} \leq 0,x_{2}\leq 0\\
 \begin{pmatrix}
 [4x_{1}+3x_{2}+1 \vee 8x_{1}+2x_{2}+2] \\ [ 3x_{1}+2x_{2}-3 \vee    2x_{1}+4x_{2}-1]
 \end{pmatrix}; \\ \text { if } x_{1} \geq 0,x_{2}\leq 0\\
 \end{cases}
 $\\ \\
Selection of descent direction at $ (1,-1) $.\\
Here $ \nabla \breve{F}(1,-1)= \begin{pmatrix}
[2,8] & [-3, -2]
\end{pmatrix} $.\\
$ \overline{d_1}= -8+6t_1 $, $ \overline{d_2}= 2+t_2 $, with $ t_1,t_2 \in [0,1] $. Choose $ t_1=\dfrac{5}{6} $ and $ t_2=0 $.
Therefore $ \overline{d} = \begin{pmatrix}
-3\\2
\end{pmatrix}$. Then $ \sum_{i=1}^2 \overline{d}_i \odot \breve{g}_i (1.-1)= [-30,-10] \prec \breve{0}$.
From Theorem \ref{desdir}, $\overline{ d} $ is a descent direction at $(1,-1)$.\\\\
In this example a descent sequence of intervals $\{\breve{F}(x^k)\}$ is generated starting with randomly chosen initial point.  At every iteration, $\alpha_k$ is selected according to \eqref{alp}. $\overline{d}^{k}$ is determined using Remark \ref{dis}. All the iterations are summarized in the following table.
\begin{table*}[!t]
		\caption{Iterations with $ x^0=(1,1)' $ }\label{tab1}
		%\scriptsize
		\begin{center}
		\begin{tabular}{|c|c|c|c|c|}
			\hline
			$k$ & $x^{k}$ & $ \breve{F}(x^{k}) $  & $ \overline{d}^{k} $ & $ \alpha_{k} $\\
			\hline
$0$ & $(1.00000, 1.00000)$ & $[5.00000, 8.00000]$  & $(-13.00000, -4.00000)$& $0.10703$\\ \hline
$1$ & $(-0.39141, 0.57188)$ & $[-1.90307, -0.77751]$  & $(-0.58438, 1.88672)$& $0.37676$\\ \hline
$2$ & $(-0.61157, 1.28271)$ & $[-2.63791, -1.06984]$  & $(-1.11913, -0.34227)$& $0.09160$\\ \hline
$3$ & $(-0.71409, 1.25136)$ & $[-2.69151,-1.16687]$  & $(-0.64637, -0.00000)$ & $0.00801$ \\ \hline
$4$ & $(-0.71926, 1.25136)$ & $[-2.69162,-1.17016]$  & $(-0.00000, -0.00000)$ & $ - $ \\ \hline
\end{tabular}
\end{center}
\end{table*}
From the Table \ref{tab1}, one may observe that $\breve{F}(x^{4})\preceq \breve{F}(x^{3})\preceq \breve{F}(x^{2})\preceq \breve{F}(x^{1})\preceq \breve{F}(x^{0})$. The iterative process terminates at $ x^4 $ since $ \parallel \overline{d}^k \parallel< \epsilon $ where $ \epsilon = 10^{-5} $. \\\\
\textbf{Justification of critical point:}\\
To justify $ x^4=(-0.71926, 1.25136) $, generated in Table \ref{tab1} as a critical point, from Definition \ref{critic}, it is enough to show that $ \sum_{i=1}^2 d_i \odot \breve{g}_{i}(x^4) \prec \breve{0} $ has no solution for any non zero $ d \in \mathbf{R}^2 $.
$\breve{g}(x^4)=\begin{pmatrix}
[-2.25136,2.87704] &[-0.65506, 0.56692]
\end{pmatrix}$. Therefore the interval inequality will be $$ d_1 \odot [-2.25136,2.87704] \oplus d_2 \odot [-0.65506, 0.56692]\prec \breve{0}.   $$
The above inequality can be reduced to two different system of inequalities (in real form) for different signs of $ d_1 $ and $ d_2 $.\\
\textbf{Case 1:} $ d_1 $ and $ d_2 $ are of same sign.\\
Therefore either $ d_ 1 \geq 0$, $ d_2 \geq 0 $  or $ d_1 \leq 0 $, $ d_2 \leq 0 $ with $ (d_1, d_2)\neq (0,0) $. In that case the above interval inequality reduces to
the following system of inequalities.
\begin{align*}
-2.25136 d_1  -0.65506 d_2 <0\\
2.87704 d_1+0.56692 d_2 <0
\end{align*}
For $ d_ 1 \geq 0$, $ d_2 \geq 0 $, the first inequality holds but the second one doesn't hold. \\
The second inequality satisfies the condition $ d_1 \leq 0 $, $ d_2 \leq 0$ but the first one does not.
Thus this system has no solution. \\ 

\textbf{Case 2:} $ d_1 $ and $ d_2 $ are of opposite sign.\\
Therefore either $ d_ 1 \geq 0$, $ d_2 \leq 0 $  or $ d_1 \leq 0 $, $ d_2 \geq 0 $ with $ (d_1, d_2)\neq (0,0) $. In that case the above interval inequality reduces to
the following system of inequalities.
\begin{align*}
2.87704 d_1  -0.65506 d_2 <0  \\ 
-2.25136 d_1+0.56692 d_2 <0 
\end{align*}
Here the first inequality holds for $ d_1 \leq 0 $, $ d_2 \geq 0 $ but the second inequality doesn't hold for the same condition. Similarly in case of $ d_ 1 \geq 0$, $ d_2 \leq 0 $, second inequality holds but the second inequality cannot hold. Hence the above system has no solution.\\ \\
From the above cases we can conclude that $ \nexists$ non zero $d \in \mathbf{R}^2 $ such that $ \sum_{i=1}^2 d_i \odot \breve{g}_{i}(x^4) \prec \breve{0} $. Hence  $ x^4 $ is a critical point of $ \breve{F} $.\\

Since $\breve{F}(x)$ is a set valued mapping so critical point of $\breve{F}(x)$ is not unique. Starting with different initial points one may have different critical points.
\end{ex}

\section{Conclusion and future scope}
The set of intervals is partially ordered. So it is not always easy to create a descent sequence for any general interval valued function. Objective of this article is to develop an iterative process to construct a descent sequence of intervals which provides a critical point of a function with bounded parameters as intervals. Some natural questions raise after the theoretical discussions of this article.
 Though $\{\breve{F}(x^k)\}$ is a descent sequence of intervals,  convergence of $\{x^k\}$ to the critical point can be guaranteed only with several assumptions on $\breve{F}$. This part is not studied here and kept for future research. The article focuses only on the iterative process.  Exact line search technique is used to decide the step length. Determination of suitable step length satisfying conditions \eqref{alp} is cumbersome for complex functions. There are several inexact line search methods for selection of step length, which may be used to justify the convergence, which remains the scope of the present contribution.

\bibliographystyle{plainnat}
\bibliography{ph4}

\begin{thebibliography}{18}
\providecommand{\natexlab}[1]{#1}
\providecommand{\url}[1]{\texttt{#1}}
\expandafter\ifx\csname urlstyle\endcsname\relax
  \providecommand{\doi}[1]{doi: #1}\else
  \providecommand{\doi}{doi: \begingroup \urlstyle{rm}\Url}\fi

\bibitem[Bhurjee and Padhan(2016)]{bhurjee2016optimality}
Ajay~Kumar Bhurjee and Saroj~Kumar Padhan.
\newblock Optimality conditions and duality results for non-differentiable
  interval optimization problems.
\newblock \emph{Journal of Applied Mathematics and Computing}, 50\penalty0
  (1-2):\penalty0 59--71, 2016.

\bibitem[Bhurjee and Panda(2012)]{bhurjee2012efficient}
Ajay~Kumar Bhurjee and Geetanjali Panda.
\newblock Efficient solution of interval optimization problem.
\newblock \emph{Mathematical Methods of Operations Research}, 76\penalty0
  (3):\penalty0 273--288, 2012.

\bibitem[Chalco-Cano et~al.(2011)Chalco-Cano, Rom{\'a}n-Flores, and
  Jim{\'e}nez-Gamero]{chalco2011generalized}
Yurilev Chalco-Cano, Heriberto Rom{\'a}n-Flores, and Mar{\'\i}a-Dolores
  Jim{\'e}nez-Gamero.
\newblock Generalized derivative and $\pi$-derivative for set-valued functions.
\newblock \emph{Information Sciences}, 181\penalty0 (11):\penalty0 2177--2188,
  2011.

\bibitem[Chalco-Cano et~al.(2013)Chalco-Cano, Rufi{\'a}n-Lizana,
  Rom{\'a}n-Flores, and Jim{\'e}nez-Gamero]{chalco2013calculus}
Yurilev Chalco-Cano, Antonio Rufi{\'a}n-Lizana, Heriberto Rom{\'a}n-Flores, and
  Mar{\'\i}a-Dolores Jim{\'e}nez-Gamero.
\newblock Calculus for interval-valued functions using generalized hukuhara
  derivative and applications.
\newblock \emph{Fuzzy Sets and Systems}, 219:\penalty0 49--67, 2013.

\bibitem[Dimitrova et~al.(1992)Dimitrova, Markov, and
  Popova]{dimitrova1992extended}
NS~Dimitrova, SM~Markov, and ED~Popova.
\newblock Extended interval arithmetics: new results and applications.
\newblock \emph{Computer Arithmetics and Enclosure Methods}, pages 225--232,
  1992.

\bibitem[Kyurkchiev and Markov(2016)]{Kyurkchiev2016}
Nikolay Kyurkchiev and Svetoslav Markov.
\newblock On the hausdorff distance between the heaviside step function and
  verhulst logistic function.
\newblock \emph{Journal of Mathematical Chemistry}, 54\penalty0 (1):\penalty0
  109--119, 2016.

\bibitem[Lupulescu(2013)]{Lupulescu201350}
Vasile Lupulescu.
\newblock Hukuhara differentiability of interval-valued functions and interval
  differential equations on time scales.
\newblock \emph{Information Sciences}, 248:\penalty0 50 -- 67, 2013.

\bibitem[Markov(1977)]{markov1977extended}
SM~Markov.
\newblock Extended interval arithmetic.
\newblock \emph{CR Acad. Bulgare Sci}, 30\penalty0 (9):\penalty0 1239--1242,
  1977.

\bibitem[Markov(1979)]{markov1979calculus}
Svetoslav Markov.
\newblock Calculus for interval functions of a real variable.
\newblock \emph{Computing}, 22\penalty0 (4):\penalty0 325--337, 1979.

\bibitem[Markov(1999)]{MARKOV1999225}
Svetoslav Markov.
\newblock An iterative method for algebraic solution to interval equations.
\newblock \emph{Applied Numerical Mathematics}, 30\penalty0 (2):\penalty0 225
  -- 239, 1999.

\bibitem[Markov(2000)]{markov2000algebraic}
Svetoslav Markov.
\newblock On the algebraic properties of convex bodies and some applications.
\newblock \emph{Journal of Convex Analysis}, 7:\penalty0 129--166, 2000.

\bibitem[Markov(2004)]{MARKOV200493}
Svetoslav Markov.
\newblock On quasilinear spaces of convex bodies and intervals.
\newblock \emph{Journal of Computational and Applied Mathematics}, 162\penalty0
  (1):\penalty0 93 -- 112, 2004.

\bibitem[Markov(2005)]{markov2005quasilinear}
Svetoslav Markov.
\newblock Quasilinear spaces and their relation to vector spaces.
\newblock \emph{Electronic Journal on Mathematics of Computation}, 2:\penalty0
  1--21, 2005.

\bibitem[Osuna-G{\'o}mez et~al.(2015)Osuna-G{\'o}mez, Chalco-Cano,
  Hern{\'a}ndez-Jim{\'e}nez, and Ruiz-Garz{\'o}n]{osuna2015optimality}
R.~Osuna-G{\'o}mez, Y.~Chalco-Cano, B.~Hern{\'a}ndez-Jim{\'e}nez, and
  G.~Ruiz-Garz{\'o}n.
\newblock Optimality conditions for generalized differentiable interval-valued
  functions.
\newblock \emph{Information Sciences}, 321:\penalty0 136--146, 2015.

\bibitem[Roumen~Anguelov(2006)]{MARKOV2006}
Blagovest~Sendov Roumen~Anguelov, Svetoslav~Markov.
\newblock The set of hausdorff continuous functions— the largest linear space
  of interval functions.
\newblock \emph{Reliable Computing}, 12:\penalty0 337--363, 2006.

\bibitem[Roumen~Anguelov(2016)]{MARKOV2015}
Svetoslav~Markov Roumen~Anguelov.
\newblock Hausdorff continuous interval functions and approximations.
\newblock In \emph{Scientific Computing, Computer Arithmetic, and Validated
  Numeric}, pages 3--13, Cham, 2016. Springer International Publishing.

\bibitem[Stefanini and Bede(2009)]{stefanini2009generalized}
Luciano Stefanini and Barnabas Bede.
\newblock Generalized hukuhara differentiability of interval-valued functions
  and interval differential equations.
\newblock \emph{Nonlinear Analysis: Theory, Methods \& Applications},
  71\penalty0 (3):\penalty0 1311--1328, 2009.

\bibitem[Stefanini et~al.(2008)]{stefanini2008generalization}
Luciano Stefanini et~al.
\newblock A generalization of hukuhara difference for interval and fuzzy
  arithmetic.
\newblock \emph{Soft Methods for Handling Variability and Imprecision, in:
  Series on Advances in Soft Computing}, 48, 2008.

\end{thebibliography}
\end{document}